\documentclass[a4paper,10pt]{amsart}

\usepackage[english]{babel}
\usepackage[utf8]{inputenc}

\usepackage{amssymb}
\usepackage{amsmath}
\usepackage{amsthm}
\usepackage{bbm}

\usepackage{enumitem}
\usepackage{tikz}
\usepackage{marginnote}

\newcommand{\bbN}{\mathbb{N}}

\newcommand{\bbR}{\mathbb{R}}

\newcommand{\suchthat}{\,|\,} % separator within sets
\DeclareMathOperator{\id}{id} % identity operator
\DeclareMathOperator{\re}{Re} % real part
\newcommand{\argument}{\mathord{\,\cdot\,}} % argument operator
\newcommand{\dxInt}{\;\mathrm{d}} % differential operator with spacing (at the end of integrals)
\newcommand{\norm}[1]{\left\lVert #1 \right\rVert} % norm
\newcommand{\modulus}[1]{\left\lvert #1 \right\rvert} % modulus
\DeclareMathOperator{\dom}{D} % domain (of an unbounded operator)
\newcommand{\spec}{\sigma} % spectrum
\newcommand{\spb}{\operatorname{s}} % spectral bound
\newcommand{\gbd}{\operatorname{\omega_0}} % spectral bound
\newcommand{\Res}{\mathcal{R}} % resolvent

\newcommand{\impliesProof}[2]{``\ref{#1} $\Rightarrow$ \ref{#2}''}

\newcommand{\goesru}{\xrightarrow{\operatorname{ru}}}
\newcommand{\goesorder}{\xrightarrow{\operatorname{o}}}
\newcommand{\ob}{{\operatorname{ob}}} % order bounded

\theoremstyle{definition}
\newtheorem{definition}{Definition}[section]

\newtheorem{example}[definition]{Example}

\theoremstyle{plain}
\newtheorem{proposition}[definition]{Proposition}
\newtheorem{lemma}[definition]{Lemma}
\newtheorem{theorem}[definition]{Theorem}
\newtheorem{corollary}[definition]{Corollary}

\numberwithin{equation}{section}

\begin{document}

\title[Order continuity of semigroups]{Order boundedness and order continuity properties of positive operator semigroups}
\author{Jochen Gl\"uck}
\address[J.\ Gl\"uck]{University of Wuppertal, School of Mathematics and Natural Sciences, Gaußstr.\ 20, 42119 Wuppertal, Germany}
\email{glueck@uni-wuppertal.de}
\author{Michael~Kaplin}
\address[M.\ Kaplin]{Institute of Mathematics, Physics and Mechanics,
Jadranska 19,
1000 Ljubljana,
Slovenija, 
www.researchgate.net/profile/Michael-Kaplin}
\date{\today}
\begin{abstract}
	Relatively uniformly continuous (ruc) semigroups were recently introduced and studied 
	by Kandić, Kramar-Fijavž, and the second-named author, 
	in order to make the theory of one-parameter operator semigroups available in the setting of vector lattices, 
	where no norm is present in general. 
	
	In this article, we return to the more standard Banach lattice setting 
	-- where both ruc semigroups and $C_0$-semigroups are well-defined concepts -- 
	and compare both notions. 
	We show that the ruc semigroups are precisely those positive $C_0$-semigroups 
	whose orbits are order bounded for small times.
	
	We then relate this result to three different topics:
	(i)~equality of the spectral and the growth bound for positive $C_0$-semigroups; 
	(ii) a uniform order boundedness principle which holds for all operator families between Banach lattices;
	and (iii) a description of unbounded order convergence in terms of almost everywhere convergence
	for nets which have an uncountable index set containing a co-final sequence.
\end{abstract}

\maketitle

\section{Introduction} 
\label{section:introduction}

On Banach lattices, a rich theory of positive $C_0$-semigroups is available and is, 
while its origins date back more than 40 years 
(see for instance \cite{Nagel1986} for the state of the art in the middle of the 1980s), 
still an active area of research 
-- see for instance \cite{ArnoldCoine2023, Vogt2022} for two recent contributions 
about the growth behaviour of positive (and more generally, \emph{eventually positive}) semigroups, 
and \cite{Arora2022, Mui2023} for two recent contributions 
about the quite subtle topic of \emph{locally eventually positive} semigroups.
An excellent introduction to the theory of positive semigroups can be found in the recent book 
\cite{BatkaiKramarRhandi2017}. 

The significance of order structure and positivity in the study of $C_0$-semigroups makes it tempting 
to develop a similar theory on the more general class of vector lattices 
without any norm.
However, since many fundamental concepts in the theory of $C_0$-semigroups 
-- such as generators and their integral characterizations \cite[Chapter~II]{EngelNagel2000} -- 
are formulated by using vector-valued calculus, 
such an endeavour is only possible if one succeeds to replace these topological notions 
by lattice theoretic substitutes. 
If one does so by using, more specifically, the concept of relatively uniform convergence in vector lattices,
one arrives at the notion of \emph{relatively uniformly continuous semigroups} 
(for short: \emph{ruc semigroups}); 
these were introduced and studied in \cite{KandicKaplin2020, KaplinKramer2020}.

When one comes back to the more specific setting of Banach lattices, 
then both notions -- ruc semigroups and $C_0$-semigroups -- are defined, 
and it is natural to compare both of them.
It is very easy to see that every ruc semigroup is $C_0$; 
the main purpose of this article is to characterize under which conditions, 
conversely, a positive $C_0$-semigroup is ruc 
(Theorem~\ref{thm:characterization}).

\subsection*{Organization of the article}

In the remaining part of the introduction we briefly recall a bit of terminology and notation.
In Section~\ref{section:order-reg-sg} we prove our main result, Theorem~\ref{thm:characterization}, 
which says among other things that a positive $C_0$-semigroup is ruc if and only if 
its orbits are locally order bounded for small times; 
moreover, we briefly discuss how one can obtain several examples of ruc semigroups from this result. 
In Section~\ref{section:growth} we show that, for a ruc semigroup on a Banach lattice, 
the growth bound and the spectral bound always coincide.

The remaining two sections are not directly related to operator semigroups, 
but provide additional context for Theorem~\ref{thm:characterization}: 
in Section~\ref{section:uniform-bdd} we prove a uniform order boundedness result, 
which shows that for operator families with order bounded orbits, 
the order bound can always be chosen to satisfy a certain norm estimate.
In the final Section~\ref{section:almost-everywhere} we consider nets 
whose index sets are uncountable, but have a countable co-final subset 
(a typical example of such an index set being $(0,\varepsilon]$ with the converse of the usual order);
for those we show how a weaker version of order convergence in $L^p$ 
-- so-called \emph{unbounded order convergence} -- 
can be described by means of almost everywhere convergence, 
provided that the representatives of the functions in the net are chosen carefully.
As a consequence, ruc semigroups on $L^p$ have an almost everywhere continuity property at the time $0$, 
even if no a priori assumption is made to enable a canonical choice of representatives.

\subsection*{Order convergence and relatively uniform convergence}

We assume the reader to be familiar with the basics of vector lattice and Banach lattice theory;
standard references for this theory are, for instance, 
the monographs \cite{LuxemburgZaanen1971, MeyerNieberg1991, Schaefer1974, Zaanen1983}.

Let us recall a few order theoretic notions of convergence. 
Let $E$ be an Ar\-chi\-me\-de\-an vector lattice.
A net $(x_j)_{j \in J}$ in $E$ is said to be \emph{order convergent} to a point $x \in E$ 
-- which we denote by $x_j \goesorder x$ -- 
if there exists a decreasing net $(y_h)_{h \in H}$ in $E_+$ which has infimum $0$ and satisfies the following property:
for each $h \in H$ there exists $j_0 \in J$ such that $\modulus{x_j - x} \le y_h$ for all $j \succeq j_0$.
The definition of order convergence is a bit subtle, in particular since a slightly different (and non-equivalent) definition 
occurs in some places in the literature; 
a detailed comparison between the two definitions is given in \cite{AbramovichSirotkin2005}.

A net $(x_j)_{j \in J}$ in $E$ is said to be \emph{relatively uniformly convergent} 
(for short: \emph{ru-convergent}) 
to a point $x \in E$ -- which we denote by $x_j \goesru x$ -- if there exists a vector $u \in E_+$
such that the following holds: 
for every number $\varepsilon > 0$ there exists an index $j_0 \in J$ 
such that $\modulus{x_j - x} \le \varepsilon u$ for all $j \succeq j_0$; 
in this case the vector $u$ is called a \emph{regulator} of the relatively uniform convergence 
of the net $(x_j)_{j \in J}$ to $x$.

Due to the Archimedean property of $E$, limits of order convergent or relatively uniformly convergent nets 
are uniquely determined.
It is clear that relatively uniform convergence implies order convergence
(with the same limit).
If $E$ is a Banach lattice with order continuous norm, the converse implication also holds 
\cite[Proposition~3]{DabboorasadEmelyanovMarabeh2018}.

\subsection*{Relatively uniformly continuous (ruc) semigroups}

A family of positive linear operators $(T_t)_{t \ge 0}$ in the Archimedean vector lattice $E$ 
is called a \emph{positive operator semigroup} if $T_0 = \id_E$ and $T(s+t) = T(s)T(t)$ for all $s,t \in [0,\infty)$.
Such a positive operator semigroup is called \emph{relatively uniformly continuous}, for short \emph{ruc}, 
if one has $T(t)x \goesru x$ as $t \downarrow 0$ for each $x \in E$; 
here, the notation ``$T(t)x \goesru x$ as $t \downarrow 0$'' is meant in the sense 
that the net $(T(t)x)_{t \in (0,\infty)}$ ru-converges to $x$, 
where the index set $(0,\infty)$ is directed conversely to the order inherited from $\bbR$.
This property is equivalent to a relatively uniform continuity property of each orbit over the entire time axis,
see \cite[Definition~3.3 and Proposition~3.5]{KandicKaplin2020}.
Relatively uniformly continuous semigroups were introduced in \cite{KandicKaplin2020}, 
and generators of such semigroups were studied, 
by using a relatively uniform version of calculus with values in vector lattices, in \cite{KaplinKramer2020}.

If $E$ is a Banach lattice, then relatively uniform convergence obviously implies norm convergence 
and hence every positive operator semigroup on a Banach lattice which is ruc is also a $C_0$-semigroup.
We assume that the reader is familiar with the basics of $C_0$-semigroup theory; 
an in-depth treatment of this topic can, for instance, be found in the monographs \cite{EngelNagel2000, Pazy1983}.

\subsection*{Further notation}

We use the notation $\bbN := \{1,2,3,\dots\}$ and $\bbN_0 := \bbN \cup \{0\}$.

\section{Order regularity properties of positive $C_0$-semigroups} 
\label{section:order-reg-sg}

Or main result is the following theorem which characterizes the ruc property 
of positive $C_0$-semigroups by local order boundedness of the orbits. 
The theorem also shows that, for positive $C_0$-semigroups, 
relatively uniform continuity is the same as order continuity at the time $0$.

\begin{theorem}[Characterizsation of ruc $C_0$-semigroups]
	\label{thm:characterization}
	Let $T = (T(t))_{t \geq 0}$ be a positive $C_0$-semigroup on a Banach lattice $E$.
	The following assertions are equivalent:
	\begin{enumerate}[label=\upshape(\roman*)]
		\item\label{thm:characterization:itm:ruc} 
		For each $x \in E$ one has $T(t)x \goesru x$ as $t \downarrow 0$, 
		i.e.\ the semigroup $T$ is relatively uniformly continuous.
		
		\item\label{thm:characterization:itm:oc} 
		For each $x \in E$ one has $T(t)x \goesorder x$ as $t \downarrow 0$. 
		
		\item\label{thm:characterization:itm:ob-ind} 
		For each $x \in E$ there exists $t_x > 0$ such that the set 
		$\{T(t)x: \, t \in [0,t_x]\}$ is order bounded in $E$.
		
		\item\label{thm:characterization:itm:ob-unif} 
		For each $x \in E$ and each $t_0 > 0$ the set 
		$\{T(t)x: \, t \in [0,t_0]\}$ is order bounded in $E$.
	\end{enumerate}
\end{theorem}

The crucial implication in the proof is to get from order boundedness of the orbits 
to the relatively uniform continuity.
It is a well-known technique, for instance in proofs of pointwise ergodic theorems, 
that relatively uniform convergence of the orbits of an operator net can be derived 
from the same property on a dense subspace if all orbits are order bounded.
(In this context, note that pointwise ergodic theorems can be formulated in an order theoretic language, 
as was observed very early by Nakano \cite{Nakano1948}.)
Thus, the essence of Theorem~\ref{thm:characterization} is to show relatively uniform continuity 
at the time $0$ for the orbits of all points in a dense set $D$. 
This is done in part~\ref{lem:ru-continuous-on-domain:itm:ru-continuous} of the following lemma.

\begin{lemma}
	\label{lem:ru-continuous-on-domain}
	Let $T = (T(t))_{t \geq 0}$ be a positive $C_0$-semigroup on a Banach lattice $E$ 
	with generator $A: E \supseteq \dom(A) \to E$. 
	Then the following holds:
	\begin{enumerate}[label=\upshape(\alph*)]
		\item\label{lem:ru-continuous-on-domain:itm:order-bdd} 
		There exists a number $\omega \in \bbR$ with the following property: 
		for each $y \in \dom(A)$ there exists $0 \leq z \in \dom(A)$ 
		such that $\modulus{T(t)y} \leq e^{wt}z$ for all $t \geq 0$.
		
		\item\label{lem:ru-continuous-on-domain:itm:ru-continuous} 
		For each $y \in \dom(A)$ one has $T(t)y \goesru y$ 
		with respect to some regulator $0 \le z \in \dom(A)$ as $t \downarrow 0$.
	\end{enumerate}
\end{lemma}

\begin{proof}
	We may, and shall, assume throughout the proof that the growth bound of the semigroup 
	is strictly negative
	(the definition of the growth bound is recalled 
	at the beginning of Section~\ref{section:growth}).
	We will use that the resolvent operator $\Res(0,A) := (-A)^{-1}$ 
	is then equal to the Laplace integral $\int_0^\infty T(s) \dxInt s$
	(which is to be understood in the strong sense), 
	and is thus positive.
	
	\ref{lem:ru-continuous-on-domain:itm:order-bdd}
	Since we assumed the growth bound of the semigroup to be strictly negative, 
	we will see that we can take $\omega$ to be equal to $0$. 

	Let $y \in \dom(A)$ and
	define $z := \Res(0,A) \modulus{Ay} \in \dom(A)$. 
	For every $t \geq 0$ we then have
	\begin{align}
		\label{eq:lem:ru-continuous-on-domain:itm:order-bdd:integral}
		\begin{split}
			\modulus{T(t)y} 
			& =
			\modulus{\Res(0,A) T(t) (-A) y}
			=
			\modulus{
				\int_0^\infty T(s+t) (-A) y \dxInt s
			}
			\\
			& \leq 
			\int_t^\infty T(s) \modulus{A y} \dxInt s 
			\leq \Res(0,A) \modulus{A y} 
			= z,
		\end{split}
	\end{align}
	as claimed.
	
	\ref{lem:ru-continuous-on-domain:itm:ru-continuous}
	First consider an arbitrary vector $y \in \dom(A^2)$. 
	For such $y$ we have $Ay \in \dom(A)$, so it follows 
	from~\eqref{eq:lem:ru-continuous-on-domain:itm:order-bdd:integral} that
	$\modulus{T(t)Ay} \le \Res(0,A) \modulus{A^2 y}$ for all $t \geq 0$,
	which in turn yields
	\begin{align}
		\label{eq:lem:ru-continuous-on-domain:itm:ru-continuous:dom-A-square}
		\modulus{T(t)y - y} 
		= 
		\modulus{\int_0^t T(s) Ay \dxInt s} 
		\le 
		t \; \Res(0,A) \modulus{A^2 y}.
	\end{align}
	Now fix a vector $y \in \dom(A)$.
	Since $\dom(A)$ is norm-dense in $E$, there exists a sequence $(x_n)$ in $\dom(A)$ 
	that converges to $x := -Ay \in E$.
	After replacing $(x_n)$ with a subsequence we may even assume that $x_n \goesru x$. 
	Since the resolvent operator $\Res(0,A)$ is positive, it follows that
	$y_n := \Res(0,A)x_n \goesru \Res(0,A)x = y$ with respect to a regulator $0 \le u \in \dom(A)$.	
	For each index $n$ one has $y_n \in \dom(A^2)$, the series 
	\begin{align*}
		v := \sum_{n=1}^\infty \frac{\modulus{A^2 y_n}}{2^n(\norm{A^2 y_n} + 1)}
	\end{align*}
	is absolutely convergent in the Banach space $E$, and $v \ge 0$.
	Since every $y_n$ is in $\dom(A^2)$ 
	and $\Res(0,A) \modulus{A^2 y_n}$ is dominated by an ($n$-dependent) multiple of $\Res(0,A)v$,
	it follows from~\eqref{eq:lem:ru-continuous-on-domain:itm:ru-continuous:dom-A-square} that, 
	for each $n$, one has $T(t)y_n \goesru y_n$ as $t \downarrow 0$ 
	with respect to the regulator $\Res(0,A)v$.
	
	To finally show that $T(t)y \goesru y$, 
	let $\varepsilon > 0$. 
	We choose an index $n$ such that 
	$\modulus{y_n-y} \le \varepsilon u$.
	According to the previous paragraph there exists a time $t_0 > 0$ such that
	$\modulus{T(t) y_n - y_n} \le \varepsilon \Res(0,A)v$ for all $t \in (0,t_0]$. 
	For each such $t$ we thus have
	\begin{align*}
		\modulus{T(t) y - y}
		& \le 
		\modulus{T(t) y - T(t) y_n} + \modulus{T(t) y_n - y_n} + \modulus{y_n - y}
		\\
		& \le 
		\varepsilon T(t) u + \varepsilon \Res(0,A)v + \varepsilon u 
		\le 
		\varepsilon \big( \Res(0,A) \modulus{A u} + \Res(0,A)v + u \big),
	\end{align*}
	where we again used~\eqref{eq:lem:ru-continuous-on-domain:itm:order-bdd:integral} 
	for the last inequality.
	So we showed that $T(t)y \goesru y$ with respect to the regulator 
	$\Res(0,A) \modulus{A u} + \Res(0,A)v + u \in \dom(A)$.
\end{proof}

\begin{proof}[Proof of Theorem~\ref{thm:characterization}]
	\impliesProof{thm:characterization:itm:ruc}{thm:characterization:itm:oc}
	This implication is clear since relatively uniform convergence of nets 
	implies order convergence to the same limit.
	
	\impliesProof{thm:characterization:itm:oc}{thm:characterization:itm:ob-ind}
	This implication follows readily from the definition of order convergence.
	
	\impliesProof{thm:characterization:itm:ob-ind}{thm:characterization:itm:ruc}
	Let $x \in E$. 
	Since $\dom(A)$ is norm dense in $E$, there exists a sequence $(x_n)$ in $\dom(A)$ 
	that converges to $x$ in norm.
	By replacing $(x_n)$ with a subsequence that converges sufficiently fast, 
	we may assume that even $x_n \goesru x$. 
	For each fixed index $n$ it follows from 
	Lemma~\ref{lem:ru-continuous-on-domain}\ref{lem:ru-continuous-on-domain:itm:ru-continuous} 
	that $T(t)x_n \goesru x_n$ as $t \downarrow 0$. 
	Moreover, as $E$ is a Banach lattice, we can find a common regulator $u \in E_+$ 
	for those countably many ru-convergences
	(indeed, one can for instance take $u$ to be a weighted sum of the regulators $u_n$ 
	of the ru-convergences $T(t)x_n \goesru x_n$); 
	moreover, by making $u$ still larger we can also achieve 
	that it is a regulator of the ru-convergence $x_n \to x$.
	According to~\ref{thm:characterization:itm:ob-ind} there is a time $t_u > 0$ 
	and a vector $v \in E_+$ such that $T(t)u \le v$ for all $t \in [0, t_u]$. 
	
	Now let $\varepsilon > 0$. 
	Choose an integer $n \ge 1$ such that $\modulus{x_n-x} \le \varepsilon u$ 
	and then choose $t_0 > 0$ such that $\modulus{T(t)x_n - x_n} \le \varepsilon u$ for all $t \in (0,t_0]$. 
	For $t \in (0, t_u \land t_0]$ we thus have
	\begin{align*}
		\modulus{T(t) x - x}
		\le 
		\modulus{T(t)x - T(t)x_n} + \modulus{T(t)x_n - x_n} + \modulus{x_n - x}
		\le 
		\varepsilon(v + 2u),
	\end{align*}
	which proves $T(t)x \goesru x$ as $t \downarrow 0$.
	
	\impliesProof{thm:characterization:itm:ruc}{thm:characterization:itm:ob-unif}
	This was proved in \cite[Proposition~3.4]{KandicKaplin2020}.
	
	\impliesProof{thm:characterization:itm:ob-unif}{thm:characterization:itm:ob-ind} 
	This implication is obvious.
\end{proof}

Let us now discuss a few examples.
Examples~\ref{exa:translation} and~\ref{exa:heat-semigroup}\ref{exa:heat-semigroup:itm:reflexive} 
have already been discussed in \cite{KandicKaplin2020}; 
still we include them here since they are essential to provide some intuition about ruc semigroups
on Banach lattices.

Recall that a Banach lattice $E$ is called an \emph{AM-space} 
if $\norm{x \lor y} = \norm{x} \lor \norm{y}$ for all $x,y \in E_+$.
For instance, if $L$ is a locally compact Hausdorff space, 
then the space $C_0(L)$ of continuous real-valued functions on $L$ which vanish at infinity is, 
with respect to the pointwise order and the sup norm, an AM-space.
An important property of AM-spaces is that every relatively compact set in such a space 
is order bounded; 
this is essentially a consequence of the Arzelà--Ascoli theorem,
see for instance \cite[Proposition~II.7.6 on p.\,106]{Schaefer1974} for details.
This observation yields the following example class of ruc semigroups.

\begin{example}[Semigroups on AM-spaces]
	\label{exa:am-spaces}
	Let $E$ be an AM-space and let $T = (T(t))_{t \ge 0}$ be a positive $C_0$-semigroup on $E$. 
	Then $T$ is relatively uniformly continuous. 
	
	Indeed, for every $x \in E$ the local orbit $\{T(t)x : \, t \in [0,1]\}$ is compact and thus,
	as $E$ is an AM-space, order bounded. 
	Thus, the implication 
	\ref{thm:characterization:itm:ob-unif}$\Rightarrow$\ref{thm:characterization:itm:ruc} 
	in Theorem~\ref{thm:characterization} 
	shows that $T$ is a ruc-semigroup.
\end{example}

A very simple class of semigroups that are not ruc are \emph{translation} (or \emph{shift}) semigroups;
here is a concrete example:

\begin{example}[Translation semigroup]
	\label{exa:translation}
	Let $p \in [1,\infty)$.
	Then the \emph{left shift semigroup} $T = (T(t))_{t \ge 0}$ on $L^p(\bbR)$, 
	which is given by
	\begin{align*}
		T(t)f = f(\argument + t)
	\end{align*}
	for all $f \in L^p(\bbR)$ and all $t \ge 0$, 
	is not relatively uniformly continuous; 
	a detailed argument for this can, for instance, 
	be found in \cite[Example~3.12]{KandicKaplin2020}. 
	
	However, for every function $f$ in the Sobolev space $W^{1,p}(\bbR)$ 
	one has $T(t)f \goesru f$ as $t \downarrow 0$
	(even with respect to a regulator from $W^{1,p}(\bbR)$);
	this follows from Lemma~\ref{lem:ru-continuous-on-domain}\ref{lem:ru-continuous-on-domain:itm:ru-continuous} 
	since $W^{1,p}(\bbR)$ is the domain of the generator of $T$.
\end{example}

As opposed to the translation semigroup, the heat semigroup on $L^p(\bbR^d)$ is indeed ruc 
if $p \in (1,\infty)$. 
We discuss this in more detail in the following example.

\begin{example}[The heat semigroup for $1 \le p < \infty$]
	\label{exa:heat-semigroup}
	Let $p \in [1,\infty)$ and let $T = (T(t))_{t \ge 0}$ denote the heat semigroup on $L^p(\bbR^d)$ 
	which is generated by the Laplace operator;
	more concretely, one has $T(t)f = k_t \star f$ 
	for each $f \in L^p(\bbR^d)$ and each $t > 0$, where $\star$ denotes the convolution 
	and the \emph{heat kernel} $k_t$ is given by
	\begin{align*}
		k_t(x) = \frac{1}{(4\pi t)^{d/2}} \exp\Big(-\frac{\norm{x}^2}{4t}\Big)
	\end{align*}
	for all $x \in \bbR^d$ and every $t > 0$.
	Whether the heat semigroup $T$ is ruc depends on $p$: 
	\begin{enumerate}[label=(\alph*)]
		\item\label{exa:heat-semigroup:itm:reflexive} 
		For $p \in (1,\infty)$ the heat semigroup $T$ is ruc. 
		Indeed, it was shown in a classical paper by Stein \cite[Corollary~2]{Stein1961}
		that every orbit of $T$ is order bounded in $L^p(\bbR^d)$ (even over the entire time axis $[0,\infty)$);
		so the ruc property of $T$ follows from Theorem~\ref{thm:characterization}.
		
		\item\label{exa:heat-semigroup:itm:p=1} 
		For $p=1$ the heat semigroup is not ruc. 
		We prove this in Proposition~\ref{prop:heat-semigroup-on-L_1} 
		in Section~\ref{section:uniform-bdd}, 
		after we have shown a uniform order boundedness result in Theorem~\ref{thm:unif-order-bdd}.
	\end{enumerate}
\end{example}

A more general class of semigroups that are ruc can be derived from the heat semigroup in the following way:

\begin{example}[Semigroups with a Gaussian estimate]
	\label{exa:gaussian-estimates}
	Let $p \in (1,\infty)$ and let $T$ denote the heat semigroup on $L^p(\bbR^d)$, 
	as in Example~\ref{exa:heat-semigroup}\ref{exa:heat-semigroup:itm:reflexive}.
	
	Let $\Omega \subseteq \bbR^d$ be a non-empty open set and let $S = (S(t))_{t \in [0,\infty)}$
	be a positive $C_0$-semigroup on $L^p(\Omega)$. 
	We interpret $L^p(\Omega)$ as a subspace of $L^p(\bbR^d)$ 
	by extending all functions in $L^p(\Omega)$ with the value $0$ outside of $\Omega$.
	The semigroup $S$ is said to \emph{satisfy a Gaussian estimate} if there exist numbers $b > 0$ and $c \ge 1$ 
	such that 
	\begin{align*}
		S(t)f \le c T(bt) f 
		\qquad \text{for all } t \in [0,1] \text{ and all } f \in L^p(\Omega).
	\end{align*}
	For more information on Gaussian estimates we refer the reader
	for instance to \cite[Section~7.4]{Arendt2004}.
	Under very mild regularity conditions semigroups generated by second order differential operators on $\Omega$ 
	with classical boundary conditions 
	(such as, for instance, Dirichlet or Neumann boundary conditions)
	satisfy Gaussian estimates, see for instance \cite[Theorem~6.10 on p.\,170]{Ouhabaz2005}.
	
	Now assume that $S$ satisfies a Gaussian estimate.
	Since all orbits of the heat semigroup on $L^p(\bbR^d)$ are order bounded 
	(as discussed in Example~\ref{exa:heat-semigroup}\ref{exa:heat-semigroup:itm:reflexive})
	it follows that all orbits of $S$ are order bounded in $L^p(\Omega)$ over the time interval $[0,1]$. 
	Hence, $S$ is relatively uniformly continuous according 
	to Theorem~\ref{thm:characterization}.
\end{example}

\section{Growth and spectral bound of ruc semigroups}
\label{section:growth}

Let $T = (T(t))_{t \ge 0}$ be a $C_0$-semigroup on a Banach lattice $E$, 
with generator $A$. 
Recall that the \emph{spectral bound} of $A$ is defined as 
\begin{align*}
	\spb(A) := \sup \{\re \lambda : \, \lambda \in \spec(A) \} \in [-\infty,\infty),
\end{align*}
where $\spec(A)$ denotes the spectrum of $A$;
more precisely speaking, it actually denotes the spectrum of the complex extension of $A$ 
to the complexification of the Banach lattice $E$
-- throughout this section we will tacitly work on the complexification of the given Banach lattice 
whenever we use spectral theory.
Let $\gbd(T)$ denote the growth bound of the semigroup $T$, i.e., 
\begin{align*}
	\gbd(T) 
	:= 
	\inf 
	\big\{
		\omega \in \bbR 
		: \, 
		\exists M \ge 1 \text{ such that } \norm{T(t)} \le M e^{\omega t} \text{ for all } t \ge 0 
	\big\}
	,
\end{align*}
which is also in $[-\infty,\infty)$. 
One always has $\spb(A) \le \gbd(T)$, but equality does not hold, in general 
-- even if the semigroup is positive. 
See for instance \cite[Theorem~5.1.11]{ArendtBattyHieberNeubrander2011} for a counterexample.

However, on some classes of Banach lattices, positivity of a $C_0$-semigroup 
does indeed imply $\spb(A) = \gbd(T)$: 
this is the case on all $L^p$-spaces, 
where the result is due to Weis \cite{Weis1995, Weis1998} 
and where a very short proof was recently given by Vogt \cite{Vogt2022};
it is also true on AM-spaces, 
where the result was originally shown by Batty and Davies \cite[Theorem~4]{BattyDavies1983}
(in fact, they even showed it on a more general class of ordered Banach spaces which, 
when restricted to the Banach lattice case, gives precisely the class of AM-spaces)
and where a new and very short proof was recently given 
by Arora and the first-named author \cite[Theorem~1]{AroraGlueck2022}.

The main ingredient in the proof of \cite[Theorem~1]{AroraGlueck2022} was local order boundedness 
of the orbits of the semigroup, which is automatic on AM-spaces 
(see Example~\ref{exa:am-spaces}). 
Hence, it is not particularly surprising that the same result remains true for ruc semigroups 
on general Banach lattices. 
The proof is essentially the same as for \cite[Theorem~1]{AroraGlueck2022}, 
but we include the details for the convenience of the reader 
(and as the argument is very short anyway): 

\begin{theorem}[Spectral bound equals growth bound]
	Let $T = (T(t))_{t \ge 0}$ be a positive $C_0$-semigroup on a Banach lattice $E$ 
	and assume that $T$ is relatively uniformly continuous. 
	Then $\spb(A) = \gbd(T)$.
\end{theorem}

\begin{proof}
	Assume for a contradiction that $\spb(A) < \gbd(T)$.
	By rescaling the semigroup we can thus assume that $\spb(A) < 0 < \gbd(T)$.
	We now show that the orbit of every $x \in E_+$ (and hence of every $x \in E$) is norm bounded;
	due to the uniform boundedness principle this proves that $T$ is norm bounded, 
	which is a contradiction to $0 < \gbd(T)$.
	
	So fix $x \in E_+$. 
	As $T$ is an ruc semigroup, there exists, 
	according to Theorem~\ref{thm:characterization}, a vector $z \in E_+$ 
	such that $0 \le T(t)x \le z$ for all $t \in [0,1]$.
	Hence we obtain for each $t \ge 1$ 
	\begin{align*}
		0 
		& \le 
		T(t)x 
		= 
		\int_0^1 T(t) x \dxInt s 
		= 
		\int_0^1 T(t-s) T(s) x \dxInt s 
		\\
		& \le 
		\int_{t-1}^t T(s) z \dxInt s 
		\le 
		\int_0^\infty T(s) z \dxInt s
		= 
		\Res(0,A) z;
	\end{align*}
	here we used that, since $\spb(A) < 0$ and the semigroup $T$ is positive, 
	the integral $\int_0^\infty T(s) z \dxInt s$ converges as an improper Riemann integral to $\Res(0,A)z$, 
	see for instance \cite[Theorem~12.7]{BatkaiKramarRhandi2017} or \cite[Theorem~5.2.1]{ArendtBattyHieberNeubrander2011}.
	Hence, the orbit of $x$ under $T$ is indeed norm bounded, as claimed.
\end{proof}

An argument close in spirit to the above one 
was recently also used in the proof of \cite[Theorem~4.6]{ArnoldCoine2023} 
to derive a growth estimate for Kreiss bounded positive $C_0$-semigroups on AM-spaces; 
however, this proof includes various further techniques, 
and it does not seem clear under which conditions the proof can be transferred to ruc semigroups on general Banach lattices.

\section{Uniform order boundedness}
\label{section:uniform-bdd}

Assertion~\ref{thm:characterization:itm:ob-unif} in Theorem~\ref{thm:characterization} 
states that, for every fixed $t_0 > 0$, 
the set $\{T(t)x : \, t \in [0,t_0] \}$ is order bounded in $E$ for every $x \in E$. 
It is natural to ask whether this set always has an order bound 
whose norm can be controlled by $\norm{x}$.
In the following theorem we show that this is indeed the case. 
The argument is not specifically related to semigroups, though, 
so we rather consider general operator families:

\begin{theorem}[Principle of uniform order boundedness]
	\label{thm:unif-order-bdd}
	Let $E$ be real Banach space and $F$ a Banach lattice. 
	Let $I$ be a non-empty set and let $(T_i)_{i \in I}$ be a family of bounded linear operators from $E$ to $F$. 
	Assume that for each $x \in E$ the orbit $\left\{T_i x: \, i \in I\right\}$ is order bounded in $F$.
	
	Then there exists a real number $M \ge 0$ with the following property: 
	for each $x \in E$ there exists a vector $y \in F_+$ of norm $\norm{y} \le M \norm{x}$ 
	such that the orbit $\{T_i x: \, i \in I\}$ is contained in the order interval $[-y,y]$.
\end{theorem}

It seems natural to call this a \emph{uniform order boundedness principle}
(where the word ``uniform'' refers to the norm of the order bounds 
rather than to the order bounds themselves).
For the proof we need the following observation:

\begin{proposition}
	\label{prop:order-bdd-banach-lattice}
	Let $F$ be a Banach lattice and let $I$ be a non-empty set. Then the vector space
	\begin{align*}
		\ell^\ob(I; F) := \left\{u: I \to F: \;  \text{the range $u(I)$ of $u$ is order bounded} \right\}
	\end{align*}
	is a Banach lattice with respect to pointwise ordering and the norm $\norm{\argument}_\ob$ that is given by
	\begin{align*}
		\norm{u}_\ob 
		= 
		\inf
		\big\{
			\norm{y} 
			: \; 
			y \in F_+ \text{ and } u(I) \subseteq [-y,y] 
		\big\}
	\end{align*}
	for each $u \in \ell^\ob(I;F)$. Moreover, we have $\norm{u}_\infty \le \norm{u}_\ob$ for each $u \in \ell^\ob(I;F)$.
\end{proposition}
\begin{proof}
	It is easy to check that $\ell^\ob(I;F)$ is a vector lattice with respect to pointwise ordering 
	and that $\norm{\argument}_\ob$ is a norm that renders $\ell^\ob(I;F)$ a normed vector lattice. 
	Clearly we have $\norm{u}_\infty \le \norm{u}_\ob$ for all $u \in \ell^\ob(I;F)$. 
	So we only need to show completeness of the norm $\norm{\argument}_\ob$ on $\ell^\ob(I;F)$.
	
	To this end, it suffices to show that every absolutely convergent series is convergent. 
	So let $(u_n)$ be a sequence in $\ell^\ob(I;F)$ such that
	\begin{align*}
		\sum_{n=1}^\infty \norm{u_n}_\ob < \infty.
	\end{align*}
	Then we have, in particular, $\sum_{n=1}^\infty \norm{u_n}_\infty < \infty$, so for each $i \in I$ the series
	\begin{align*}
		\sum_{n=1}^\infty u_n(i)
	\end{align*}
	converges to a vector $u(i)$ in $F$. 
	We are going to show that the vector $u: I \ni i \mapsto u(i) \in F$ is an element of $\ell^\ob(I;F)$ 
	and that it is, moreover, the limit of $\sum_{n=1}^N u_n$ with respect to $\norm{\argument}_\ob$ as $N \to \infty$.
	
	For each index $n \in \bbN$ we can find a vector $y_n \in F_+$ 
	such that $\norm{y_n} \le \norm{u_n}_\ob + \frac{1}{2^n}$ and $u_n(I) \subseteq [-y_n, y_n]$. 
	Hence, the vector
	\begin{align*}
		y := \sum_{n=1}^\infty y_n
	\end{align*}
	is a well-defined element of $F_+$, 
	and the order interval $[-y,y]$ contains $u(i)$ for each $i \in I$. 
	This proves that $u $ is an element of $\ell^\ob(I;F)$.
	
	Moreover, for each $N \in \bbN$ and each $i \in I$ we have
	\begin{align*}
		\modulus{u(i) - \sum_{n=1}^N u_n(i)} \le \sum_{n=N+1}^\infty y_n := z_N \in F.
	\end{align*}
	Thus, the $\norm{\argument}_\ob$-norm of $u - \sum_{n=1}^N u_n$ 
	is dominated by $\norm{z_N}$ for each $N \in \bbN$, 
	and the latter norm tends to $0$ as $N \to \infty$. 
	This proves that $u$ is indeed the limit of the series over $(u_n)$ 
	with respect to the $\norm{\argument}_\ob$-norm.
\end{proof}

The proof of Theorem~\ref{thm:unif-order-bdd} is now an easy consequence of the closed graph theorem:

\begin{proof}[Proof of Theorem~\ref{thm:unif-order-bdd}]
	Let $\left(\ell^\ob(I;F), \norm{\argument}_\ob\right)$ be the Banach lattice 
	introduced in Proposition~\ref{prop:order-bdd-banach-lattice}. 
	Consider the linear operator $T: E \to \left(\ell^\ob(I;F), \norm{\argument}_\ob\right)$ 
	which is given by
	\begin{align*}
		Tx := (T_ix)_{i \in I}
	\end{align*}
	for each $x \in E$; this operator is well defined 
	since $(T_ix)_{i \in I}$ is an element of $\ell^\ob(I;F)$ for each $x \in E$ by assumption.
	
	It follows from the continuity of the single operators $T_i$ 
	and from the fact that the norm $\norm{\argument}_\ob$ on $\ell^\ob(I;F)$ is stronger than the sup norm, 
	that the operator $T$ has closed graph; 
	therefore, $T$ is continuous by the closed graph theorem. 
	This immediately implies the assertion of the theorem 
	(where $M$ can be chosen as any number that is strictly larger than the operator norm of $T$).
\end{proof}

Note that one can use a similar argument to prove the classical uniform boundedness theorem 
for linear operators between Banach spaces; 
one just has to replace the space $\ell^\ob(I; F)$ in the proof with the space
$\ell^\infty(I; F)$ of all bounded functions from $I$ to $F$, endowed with the sup norm.

It is a well-known phenomenon in harmonic analysis that the maximal operator associated 
to the heat semigroup is not bounded in $L^1(\bbR)$, 
see for instance the discussion in \cite[introduction on p.\,417]{NowakSjoegren2007}.
In fact, this is even true if one only considers the time interval $[0,1]$.
By means of Theorem~\ref{thm:unif-order-bdd} this implies that the heat semigroup on $L^1(\bbR^d)$ 
is not relatively uniformly continuous.
For the convenience of the reader we give all the details of the argument in the proof of the following proposition
(to which we have already referred in Example~\ref{exa:heat-semigroup}\ref{exa:heat-semigroup:itm:p=1}).

\begin{proposition}[The heat semigroup is not ruc on $L^1$]
	\label{prop:heat-semigroup-on-L_1}
	The heat semigroup $(T(t))_{t \ge 0}$ from Example~\ref{exa:heat-semigroup} is not relatively uniformly continuous 
	on the space $L^1(\bbR^d)$.
\end{proposition}

\begin{proof}
	For each time $t \ge 0$ the dual operator $T(t)'$ on $L^\infty(\bbR^d)$ 
	(which also acts by convolution with the same kernel $k_t$ as $T(t)$)
	leaves the space $C_0(\bbR^d)$ 
	(which consists of the continuous functions that vanishes at infinity, endowed with the sup norm) 
	invariant and the restriction of the operator family $(T(t)')_{t \ge 0}$ to $C_0(\bbR^d)$, 
	which we now denote by $(S(t))_{t \ge 0}$, 
	is a $C_0$-semigroup on this space. 
	
	There is a canonical embedding of $L^1(\bbR^d)$ into the dual space $C_0(\bbR^d)'$ of $C_0(\bbR^d)$, 
	and this embedding is an isometric lattice homomorphism;
	in fact, if we identify $C_0(\bbR^d)'$ with the space of finite Borel measures on $\bbR^d$ 
	by means of the Riesz representation theorem, 
	then $L^1(\bbR^d)$ can be identified with the band within $C_0(\bbR^d)'$ 
	that consists of those measures that are absolutely continuous with respect to the Lebesgue measure.
	Moreover, it is not difficult to see that 
	there exists a net $(f_j)_{j \in J}$ of normalized positive elements of $L^1(\bbR^d)$ 
	which is weak${}^*$-convergent to $\delta_0 \in C_0(\bbR^d)'$, 
	where $\delta_0$ denotes the point evaluation at the point $0 \in \bbR^d$.
	
	Now assume towards a contradiction that the heat semigroup is ruc on $L^1(\bbR^d)$. 
	Then it follows from Theorem~\ref{thm:characterization}\ref{thm:characterization:itm:ruc} 
	and~\ref{thm:characterization:itm:ob-unif} together with Theorem~\ref{thm:unif-order-bdd} 
	that there exists a number $M \ge 0$ and positive vectors $u_j \in L^1(\bbR^d)$ of norm at most $M$ 
	for each $j \in J$ such that $0 \le T(t) f_j \le u_j$ for each $j \in J$ and each $t \in [0,1]$.
	When expressed in the dual space of $C_0(\bbR^d)$ this means that
	\begin{align*}
		0 \le S(t)' f_j \le u_j
	\end{align*}
	for each $j \in J$ and each $t \in [0,1]$.
	After switching to subnets we may assume that $(u_j)$ is weak${}^*$-convergent to an element 
	$0 \le u \in C_0(\bbR^d)'$ of norm $\norm{u} \le M$.
	Since the operator $S(t)'$ is, for each $t$, continuous with respect to the weak${}^*$-topology 
	and since the positive cone in $C_0(\bbR^d)'$ is weak${}^*$-closed,
	we thus obtain
	\begin{align*}
		0 \le S(t)' \delta_0 \le u
	\end{align*}
	for each $t \in [0,1]$.
	However, we have $S(t)'\delta_0 \in L^1(\bbR^d)$ for each $t > 0$
	(in fact, one has $S(t)' \delta_0 = k_t$, where $k_t \in L^1(\bbR^d)$ 
	denotes the heat kernel defined in Example~\ref{exa:heat-semigroup}).
	If $P$ denotes the band projection on $C_0(\bbR^d)'$ onto $L^1(\bbR^d)$, 
	it thus follows that 
	\begin{align*}
		0 \le S(t)' \delta_0 = P S(t)' \delta_0 \le P u
	\end{align*}
	for all $t \in (0,1]$.
	Since $(S(t))_{t \ge 0}$ is a $C_0$-semigroup on $C_0(\bbR^d)$ 
	we have $S(t)' \delta_0 \to \delta_0$ with respect to the weak${}^*$-topology as $t \downarrow 0$.
	Therefore, $0 \le \delta_0 \le Pu$; 
	but this is a contradiction since $Pu \in L^1(\bbR^d)$.
\end{proof}

\section{Unbounded order convergence and almost everywhere convergence of nets}
\label{section:almost-everywhere}

Recall that a net $(x_j)$ in a Banach lattice $E$ is said to be \emph{unboundedly order convergent}
to a point $x \in E$ if for every $u \in E_+$ the net $(\modulus{x_j - x} \land u)$ order converges to $0$.
So in particular, if the net $(x_j)$ is order convergent to $x$, 
then it is also unboundedly order convergent to $x$.
Some recent contributions to the study of unbounded order convergence include 
\cite{DabboorasadEmelyanovMarabeh2018, LiChen2018, Taylor2019}.

Now consider an analytic positive $C_0$-semigroup $T = (T(t))_{t \ge 0}$ on an $L^p$-space 
over a $\sigma$-finite measure space.
It follows from the local Taylor series expansion 
of the semigroup that for every $f \in L^p$ representatives of the vectors $T(t)f$ 
can be simultaneously chosen for all $t \in [0,\infty)$ in a way that ensures 
that the orbit of $f$ is pointwise continuous on the time interval $(0,\infty)$; 
see for instance \cite[the lemma on p.\,72]{Stein1970} for details.
Since, for sequences in $L^p$, it is well-known (and easy to see) 
that unbounded order convergence is equivalent to almost everywhere convergence, 
one concludes that ruc continuity of analytic semigroups on $L^p$ implies 
that each orbit can be chosen to be pointwise continuous 
even on the time interval $[0,\infty)$, i.e.\ including the time $0$.
So for instance the heat semigroup considered in 
Example~\ref{exa:heat-semigroup}\ref{exa:heat-semigroup:itm:reflexive} has this property.
This almost everywhere continuity property of semigroups was also discussed 
in \cite{Stein1961} and was described in that paper as the motivation to prove the maximal inequality 
that was used in Example~\ref{exa:heat-semigroup}\ref{exa:heat-semigroup:itm:reflexive} above. 

For semigroups $T$ that are not analytic it is not clear under which conditions it is possible 
to choose the representatives of all vectors in an orbit $t \mapsto T(t)f$ in a way which gives, say, 
at least the measurability of the function $t \mapsto (Tf)(\omega)$ for almost all $\omega$.
Still, it is desirable to interpret the order continuity property that occurs in Theorem~\ref{thm:characterization}\ref{thm:characterization:itm:oc} 
in a pointwise almost everywhere manner; 
to this end we prove the following result.

\begin{theorem}[Almost everywhere convergence of nets]
	\label{thm:almost-everywhere}
	Let $(\Omega,\mu)$ be a $\sigma$-finite measure space and $p \in [1,\infty)$. 
	Let $f \in L^p(\Omega,\mu)$, let $(f_j)_{j \in J}$ be a net in $L^p(\Omega,\mu)$, and assume that the index set $J$ 
	contains a co-final sequence $(j_n)_{n \in \bbN}$. 
	Then the following assertions are equivalent:
	\begin{enumerate}[label=\upshape(\roman*)]
		\item\label{thm:almost-everywhere:itm:uoc} 
		The net $(f_j)_{j \in J}$ is unboundedly order convergent to $f$. 
		
		\item\label{thm:almost-everywhere:itm:ae} 
		There exists a representative $\hat f: \Omega \to \bbR$ of $f$ and a family $(\hat f_j)_{j \in J}$ of representatives $\hat f_j: \Omega \to \bbR$ of $f_j$ such that
		\begin{align*}
			\lim_j \hat f_j(\omega) = \hat f(\omega)
		\end{align*}
		for all $\omega \in \Omega$.
	\end{enumerate}
\end{theorem}

Note that it does not make any difference if one replaces ``for all $\omega \in \Omega$'' 
with ``for almost all $\omega \in \Omega$'' in part~\ref{thm:almost-everywhere:itm:ae}.
However, we deliberately chose the wording ``for all'' in order to stress 
that one speaks about functions $\Omega \to \bbR$ (rather than their almost everywhere equivalence classes) 
in~\ref{thm:almost-everywhere:itm:ae}.
In contrast to the case of sequences 
the existence quantifier over the representatives $\hat f_j$ 
is crucial in~\ref{thm:almost-everywhere:itm:ae}: 
if one changes each of the representatives $\hat f_j$ on a set of measure $0$, 
it might happen that $\hat f_j(\omega)$ does not converge to $f(\omega)$ for any $\omega \in \Omega$ 
since the index set $J$ can be uncountable.

For the proof of Theorem~\ref{thm:almost-everywhere} 
we use the following auxiliary result.

\begin{lemma}
	\label{lem:co-final-sequence} 
	Let $E$ be a Banach lattice with order continuous norm, let $x \in E$, 
	let $(x_j)_{j \in J}$ be an order bounded net in $E$ 
	and assume that there exists a co-final sequence $(j_n)_{n \in \bbN}$ in $J$.
	Then the following are equivalent:
	\begin{enumerate}[label=\upshape(\roman*)]
		\item\label{lem:co-final-sequence:itm:net} 
		The net $(x_j)_{j\in J}$ order converges to $x$.
		
		\item\label{lem:co-final-sequence:itm:sequences} 
		For every co-final sequence $(i_n)_{n \in \bbN}$ in $J$ 
		the sequence $(x_{i_n})_{n \in \bbN}$ order converges to $x$.
	\end{enumerate}
\end{lemma}

\begin{proof}
	\impliesProof{lem:co-final-sequence:itm:net}{lem:co-final-sequence:itm:sequences}
	This implication is clear.
	
	\impliesProof{lem:co-final-sequence:itm:sequences}{lem:co-final-sequence:itm:net}
	We may, and shall, assume that $x = 0$ and $x_j \ge 0$ for each $j \in J$. 
	For every $n \in \bbN$ we define $y_n := \bigvee_{j \succeq j_n} x_j \in E_+$; 
	moreover, let $y := \bigwedge_{n \in \bbN} y_n \in E_+$. 
	It suffices to show that $y = 0$, so assume the contrary. 
	After rescaling all elements we may than assume that $\norm{y} = 1$. 
	For every $n \in \bbN$ there exists a finite non-empty set $F_n \subseteq J$ 
	such that all $j \in F_n$ satisfy $j \succeq j_n$ and such that 
	$\norm{y_n - \bigvee_{j \in F_n} x_j} \le \frac{1}{2^{n+1}}$;
	this follows from the order continuity of the norm on $E$.
	We claim that the positive vector
	\begin{align*}
		z 
		:= 
		\bigwedge_{n \in \bbN} \bigvee_{j \in F_n} x_j
	\end{align*}
	is non-zero. 
	Indeed, since $y_n \ge y$ for each $n$, it follows that
	\begin{align*}
		z
		\ge
		\bigwedge_{n \in \bbN}\Big( y - \big(y_n - \bigvee_{j \in F_n} x_j\big) \Big) 
		=
		y - \bigvee_{n \in \bbN} \big(y_n - \bigvee_{j \in F_n} x_j\big). 
	\end{align*}
	Since $z \ge 0$ we conclude that
	\begin{align*}
		z
		\ge 
		\Big(y - \bigvee_{n \in \bbN} \big(y_n - \bigvee_{j \in F_n} x_j\big)\Big)^+
		= 
		y - \bigvee_{n \in \bbN} y \land \Big(y_n - \bigvee_{j \in F_n} x_j\Big) \ge 0.
	\end{align*}
	As one has the norm estimate
	\begin{align*}
		\Big\lVert\bigvee_{n \in \bbN} y \land \big(y_n - \bigvee_{j \in F_n} x_j\big) \Big\rVert 
		\le 
		\sum_{n \in \bbN} \big\lVert y_n - \bigvee_{j \in F_n} x_j \big\rVert
		\le 
		\frac{1}{2}
	\end{align*}
	and as $\norm{y} = 1$, it follows that $z \not= 0$, as claimed.
	
	From this one can easily construct a co-final sequence $(i_n)_{n \in \bbN}$ in $J$  
	such that $(x_{i_n})_{n \in \bbN}$ is not order convergent to $0$: 
	first list all elements of $F_1$ in arbitrary order, then all elements of $F_2$ in arbitrary order, and so on.
	The resulting sequence $(i_n)_{n \in \bbN}$ in $J$ has all required properties.
\end{proof}

\begin{proof}[Proof of Theorem~\ref{thm:almost-everywhere}]
	\impliesProof{thm:almost-everywhere:itm:uoc}{thm:almost-everywhere:itm:ae} 
	If there exists an index $\hat j \in J$ that satisfies $\hat j \succeq j$ for each $j \in J$, 
	then the claimed implication is trivial, so assume from now on that such a $\hat j$ does not exist.
	Thus we may, and shall, assume throughout the proof that, for each $n \in \bbN$, 
	one has $j_{n+1} \succeq j_n$ and $j_n \not \succeq j_{n+1}$.
	
	We choose an arbitrary representative $\hat f: \Omega \to \bbR$ of $f$. 
	As $(\Omega, \mu)$ is $\sigma$-finite, there exists a quasi-interior point $0 \le u \in L^p(\Omega)$;
	we can find a representative $\hat u: \Omega \to \bbR$ of $u$ which satisfies $\hat u(\omega) > 0$ 
	for all $\omega \in \Omega$.
	Due to~\ref{thm:almost-everywhere:itm:uoc} one has $\modulus{f_j-f} \land u \goesorder 0$.
	For every $n \in \bbN$, define 
	\begin{align*}
		g_n := \bigvee_{j \succeq j_n} \modulus{f_j-f} \land u,
	\end{align*}
	where the supremum is taken within the order complete vector lattice $L^p(\Omega,\mu)$.
	Then the sequence $g_n$ is decreasing and has infimum $0$.
	We can choose representatives $\hat g_n: \Omega \to [0,\infty)$ of the vectors $g_n$ such
	that the real sequence $\big(\hat g_n(\omega)\big)_{n \in \bbN}$ 
	is decreasing and converges to $0$ for every $\omega \in \Omega$.
	
	Now define subsets $J_n$ of $J$ for each $n \in \bbN_0$ by setting
	\begin{align*}
		J_n 
		:= 
		\big\{
			j \in J \suchthat j \succeq j_n \text{ and } j \not\succeq j_{n+1}
		\big\}
	\end{align*}
	for each $n \ge 1$ and $J_0 := J \setminus \bigcup_{n \in \bbN} J_n$. 
	Then $J$ is the disjoint union of the family $(J_n)_{n \in \bbN_0}$ 
	and we have $j_n \in J_n$ for each $n \in \bbN$.
	
	Now we choose the representatives $\hat f_j$ as follows: fix $j \in J$.
	If $j \in J_0$, let $\hat f_j: \Omega \to \bbR$ be an arbitrary representative of $f_j$.
	If $j \in J_n$ for some $n \ge 1$, then $g_n \ge \modulus{f_j-f} \land u$.
	First choose an arbitrary representative $\tilde f_j$ of $f_j$. 
	Then we have $\hat g_n \ge \modulus{\tilde f_j- \hat f} \land \hat u$ on $\Omega \setminus N_j$, 
	where $N_j \subseteq \Omega$ is some measurable set of measure $0$.
	For all $\omega \in N_j$ we now change the value $\tilde f_j(\omega)$ 
	and set it to $\hat f(\omega)$ instead; 
	this yields a new representative of $f_j$ which we call $\hat f_j$.
	Since $\hat g_n$ we chosen to take values in $[0,\infty)$ only, 
	it follows that $\hat g_n \ge \modulus{\hat f_j- \hat f} \land \hat u$ everywhere on $\Omega$.
	
	Consider a single point $\omega \in \Omega$ now and let $\varepsilon > 0$. 
	There exists $n \in \bbN$ such that $g_n(\omega) \le \varepsilon$. 
	For every $j \in J$ that satisfies $j \succeq j_n$ there exists an integer $m \ge n$ 
	such that $j \in J_m$; 
	hence,
	\begin{align*}
		\modulus{\hat f_j(\omega) - \hat f(\omega)} \land \hat u(\omega) 
		\le 
		g_m(\omega) 
		\le
		g_n(\omega) 
		\le 
		\varepsilon.
	\end{align*}
	This proves that $\modulus{\hat f_j(\omega) - \hat f(\omega)} \land \hat u(\omega) \to 0$.
	As we have $\hat u(\omega) > 0$ due to the choice of $\hat u$, 
	this implies that $\hat f_j(\omega) \to \hat f(\omega)$, as claimed. 
	
	\impliesProof{thm:almost-everywhere:itm:ae}{thm:almost-everywhere:itm:uoc} 
	For every co-final sequence $(i_n)$ in $J$ and every vector $0 \le u \in L^p(\Omega,\mu)$ 
	it is a well-known consequence of~\ref{thm:almost-everywhere:itm:ae}
	that $\modulus{f_{i_n} - f} \land u$ is order convergent to $0$.
	It thus follows from Lemma~\ref{lem:co-final-sequence} 
	that the net $\big(\modulus{f_{j} - f} \land u\big)_{j \in J}$ 
	order converges to $0$;
	this proves~\ref{thm:almost-everywhere:itm:uoc}.
\end{proof}

Applying Theorem~\ref{thm:almost-everywhere} to ruc semigroups on $L^p$-spaces 
yields the following pointwise continuity result.

\begin{corollary}[Almost everywhere continuity for ruc semigroups on $L^p$]
	\label{cor:ae-continuous}
	Let $(\Omega,\mu)$ be a $\sigma$-finite measure space and $p \in [1,\infty)$.
	Let $T = (T(t))_{t \geq 0}$ be a positive $C_0$-semigroup on $L^p(\Omega,\mu)$ 
	which is relatively uniformly continuous and let $f \in L^p(\Omega,\mu)$.
	
	Then there exists a representative $\hat f: \Omega \to \bbR$ of the vector $f$ 
	and a family $(\widehat{T(t)f})_{t \in (0,\infty)}$ 
	of representatives $\widehat{T(t)f}: \Omega \to \bbR$ of the orbit vectors $T(t)f$ such that
	\begin{align*}
		\lim_{t \downarrow 0} \widehat{T(t)f}(\omega) = \widehat{f}(\omega)
	\end{align*}
	for all $\omega \in \Omega$.
\end{corollary}

\subsection*{Acknowledgements} 

The results in Sections~\ref{section:order-reg-sg} and~\ref{section:almost-everywhere} 
of this article are based upon work from COST Action CA18232 MAT-DYN-NET, 
supported by COST (European Cooperation in Science and Technology).
Theorem~\ref{thm:characterization} and Lemma~\ref{lem:ru-continuous-on-domain} 
are, up to minor changes, contained in the second-named author's doctoral dissertation at the University of Ljubljana 
\cite[Theorem~3.13 and Proposition~3.11]{Kaplin2020}.

The authors are indebted to Sahiba Arora and to the referee for pointing out 
a number of inaccuracies in an earlier version of the manuscript.

\bibliographystyle{plain}
\bibliography{literature}

\begin{thebibliography}{10}

\bibitem{AbramovichSirotkin2005}
Yuri Abramovich and Gleb Sirotkin.
\newblock On order convergence of nets.
\newblock {\em Positivity}, 9(3):287--292, 2005.

\bibitem{Arendt2004}
Wolfgang Arendt.
\newblock Semigroups and evolution equations: functional calculus, regularity
  and kernel estimates.
\newblock In {\em Handbook of differential equations: Evolutionary equations.
  Vol. I}, pages 1--85. Amsterdam: Elsevier/North-Holland, 2004.

\bibitem{ArendtBattyHieberNeubrander2011}
Wolfgang Arendt, Charles J.~K. Batty, Matthias Hieber, and Frank Neubrander.
\newblock {\em Vector-valued {Laplace} transforms and {Cauchy} problems},
  volume~96 of {\em Monogr. Math., Basel}.
\newblock Basel: Birkh{\"a}user, 2nd ed. edition, 2011.

\bibitem{ArnoldCoine2023}
Loris Arnold and Cl{\'e}ment Coine.
\newblock Growth rate of eventually positive kreiss bounded
  {{\(C_0\)}}-semigroups on {{\(L^p\)}} and {{\(\mathcal{C}(K)\)}}.
\newblock {\em J. Evol. Equ.}, 23(1):18, 2023.
\newblock Id/No 7.

\bibitem{Arora2022}
Sahiba Arora.
\newblock Locally eventually positive operator semigroups.
\newblock {\em J. Oper. Theory}, 88(1):205--244, 2022.

\bibitem{AroraGlueck2022}
Sahiba Arora and Jochen Gl{\"u}ck.
\newblock Stability of (eventually) positive semigroups on spaces of continuous
  functions.
\newblock {\em C. R., Math., Acad. Sci. Paris}, 360:771--775, 2022.

\bibitem{BatkaiKramarRhandi2017}
Andr{\'a}s B{\'a}tkai, Marjeta Kramar~Fijav{\v{z}}, and Abdelaziz Rhandi.
\newblock {\em Positive operator semigroups. {From} finite to infinite
  dimensions}, volume 257 of {\em Oper. Theory: Adv. Appl.}
\newblock Basel: Birkh{\"a}user/Springer, 2017.

\bibitem{BattyDavies1983}
Charles J.~K. Batty and E.~Brian Davies.
\newblock Positive semigroups and resolvents.
\newblock {\em J. Oper. Theory}, 10:357--363, 1983.

\bibitem{DabboorasadEmelyanovMarabeh2018}
Yousef~A. Dabboorasad, Èduard~Y. Emelyanov, and Mohammed A.~A. Marabeh.
\newblock {{\(u\tau \)}}-convergence in locally solid vector lattices.
\newblock {\em Positivity}, 22(4):1065--1080, 2018.

\bibitem{EngelNagel2000}
Klaus-Jochen Engel and Rainer Nagel.
\newblock {\em One-parameter semigroups for linear evolution equations}, volume
  194 of {\em Grad. Texts Math.}
\newblock Berlin: Springer, 2000.

\bibitem{KandicKaplin2020}
Marko Kandi{\'c} and Michael. Kaplin.
\newblock Relatively uniformly continuous semigroups on vector lattices.
\newblock {\em J. Math. Anal. Appl.}, 489(1):23, 2020.
\newblock Id/No 124139.

\bibitem{Kaplin2020}
Michael Kaplin.
\newblock {\em Relatively uniformly continuous semigroups of positive operators
  on vector lattices}.
\newblock PhD thesis, University of Ljubljana, 2020.

\bibitem{KaplinKramer2020}
Michael Kaplin and Marjeta Kramer~Fijav{\v{z}}.
\newblock Generation of relatively uniformly continuous semigroups on vector
  lattices.
\newblock {\em Anal. Math.}, 46(2):293--322, 2020.

\bibitem{LiChen2018}
Hui Li and Zili Chen.
\newblock Some loose ends on unbounded order convergence.
\newblock {\em Positivity}, 22(1):83--90, 2018.

\bibitem{LuxemburgZaanen1971}
Wilhelmus A.~J. Luxemburg and Adriaan~C. Zaanen.
\newblock {\em Riesz spaces. {Vol}. {I}}, volume~1 of {\em North-Holland Math.
  Libr.}
\newblock Elsevier (North-Holland), Amsterdam, 1971.

\bibitem{MeyerNieberg1991}
Peter Meyer-Nieberg.
\newblock {\em Banach lattices}.
\newblock Universitext. Berlin etc.: Springer-Verlag, 1991.

\bibitem{Mui2023}
Jonathan Mui.
\newblock Spectral properties of locally eventually positive operator
  semigroups.
\newblock {\em Semigroup Forum}, 2023.
\newblock DOI: 10.1007/s00233-023-10347-0.

\bibitem{Nagel1986}
Rainer Nagel, editor.
\newblock {\em One-parameter semigroups of positive operators}, volume 1184 of
  {\em Lect. Notes Math.}
\newblock Springer, Cham, 1986.

\bibitem{Nakano1948}
Hidegoro Nakano.
\newblock Ergodic theorems in semi-ordered linear spaces.
\newblock {\em Ann. Math. (2)}, 49:538--556, 1948.

\bibitem{NowakSjoegren2007}
Adam Nowak and Peter Sj{\"o}gren.
\newblock Weak type {{\((1,1)\)}} estimates for maximal operators associated
  with various multi-dimensional systems of {Laguerre} functions.
\newblock {\em Indiana Univ. Math. J.}, 56(1):417--436, 2007.

\bibitem{Ouhabaz2005}
El~Maati Ouhabaz.
\newblock {\em Analysis of heat equations on domains}, volume~31 of {\em Lond.
  Math. Soc. Monogr. Ser.}
\newblock Princeton, NJ: Princeton University Press, 2005.

\bibitem{Pazy1983}
Amnon Pazy.
\newblock {\em Semigroups of linear operators and applications to partial
  differential equations}, volume~44 of {\em Appl. Math. Sci.}
\newblock Springer, Cham, 1983.

\bibitem{Schaefer1974}
Helmut~H. Schaefer.
\newblock {\em Banach lattices and positive operators}, volume 215 of {\em
  Grundlehren Math. Wiss.}
\newblock Springer, Cham, 1974.

\bibitem{Stein1961}
Elias~M. Stein.
\newblock On the maximal ergodic theorem.
\newblock {\em Proc. Natl. Acad. Sci. USA}, 47:1894--1897, 1961.

\bibitem{Stein1970}
Elias~M. Stein.
\newblock {\em Topics in harmonic analysis related to the {Littlewood}-{Paley}
  theory}, volume~63 of {\em Ann. Math. Stud.}
\newblock Princeton University Press, Princeton, NJ, 1970.

\bibitem{Taylor2019}
Mitchell~A. Taylor.
\newblock Unbounded topologies and {$uo$}-convergence in locally solid vector
  lattices.
\newblock {\em J. Math. Anal. Appl.}, 472(1):981--1000, 2019.

\bibitem{Vogt2022}
Hendrik Vogt.
\newblock Stability of uniformly eventually positive {{\(C_0\)}}-semigroups on
  {{\(L_p\)}}-spaces.
\newblock {\em Proc. Am. Math. Soc.}, 150(8):3513--3515, 2022.

\bibitem{Weis1995}
Lutz Weis.
\newblock The stability of positive semigroups on {{\(L_ p\)}} spaces.
\newblock {\em Proc. Am. Math. Soc.}, 123(10):3089--3094, 1995.

\bibitem{Weis1998}
Lutz Weis.
\newblock A short proof for the stability theorem for positive semigroups on
  {{\(L_p(\mu)\)}}.
\newblock {\em Proc. Am. Math. Soc.}, 126(11):3253--3256, 1998.

\bibitem{Zaanen1983}
Adriaan~C. Zaanen.
\newblock {\em Riesz spaces {II}}, volume~30 of {\em North-Holland Math. Libr.}
\newblock Elsevier (North-Holland), Amsterdam, 1983.

\end{thebibliography}

\end{document}